\documentclass[11pt,a4paper]{article}

\usepackage{soul}
\usepackage{tikz}
\usepackage{amsthm}
\usepackage{amsmath}
\usepackage{amssymb}
\usepackage{mathrsfs}
\usepackage{geometry}
\usepackage{graphicx}
\usepackage{amsfonts}
\usepackage{epstopdf}
\usepackage{placeins}
\usepackage{enumerate}
\usepackage{enumitem}
\usepackage{hyperref}
\usepackage{subcaption}

\usetikzlibrary{calc}

\numberwithin{equation}{section}

\allowdisplaybreaks

\newcommand{\md}{\mathrm{d}}

\newcommand{\R}{\mathbb{R}}

\newcommand{\Rmnum}[1]{\uppercase\expandafter{\romannumeral#1}} 

\newcommand{\mybar}[1]{\overline{#1}}
\newcommand{\mytilde}[1]{\widetilde{#1}}

\newcommand{\bbZ}{\mathbb{Z}}

\newcommand{\calB}{\mathcal{B}}

\newcommand{\calE}{\mathcal{E}}
\newcommand{\calF}{\mathcal{F}}

\newcommand{\calH}{\mathcal{H}}

\newcommand{\calK}{\mathcal{K}}

\newcommand{\calN}{\mathcal{N}}

\newcommand{\calV}{\mathcal{V}}

\def\Xint#1{\mathchoice
{\XXint\displaystyle\textstyle{#1}}%
{\XXint\textstyle\scriptstyle{#1}}%
{\XXint\scriptstyle\scriptscriptstyle{#1}}%
{\XXint\scriptscriptstyle\scriptscriptstyle{#1}}%
\!\int}
\def\XXint#1#2#3{{\setbox0=\hbox{$#1{#2#3}{\int}$ }
\vcenter{\hbox{$#2#3$ }}\kern-.6\wd0}}

\def\dashint{\Xint-}

\newtheorem{mythm}{Theorem}[section]
\newtheorem{myprop}[mythm]{Proposition}
\newtheorem{mylem}[mythm]{Lemma}

\renewcommand{\thefootnote}{}

\AtEndDocument{
\bigskip{\footnotesize%
   \textsc{Department of Mathematics, Aarhus University, 8000 Aarhus C, Denmark} \par
  \textit{E-mail address}: \texttt{yangmengqh@gmail.com}
}}

\begin{document}

\title{\Large{A direct proof of the cutoff Sobolev inequality on the Sierpi\'nski gasket}}
\author{Meng Yang}
\date{}

\maketitle

\abstract{We present a \emph{direct} proof of the cutoff Sobolev inequality on the Sierpi\'nski gasket, which has long been regarded as highly non-trivial in the context of heat kernel estimates.}

\footnote{\textsl{Date}: \today}
\footnote{The author was very grateful to Eryan Hu for introducing this topic and for reading a preliminary draft, and to Alexander Grigor'yan for reading a preliminary draft and for raising a question about the essential ingredients for a proof on the Sierpi\'nski carpet.}

\renewcommand{\thefootnote}{\arabic{footnote}}
\setcounter{footnote}{0}

\section{Introduction}

Let us recall the following classical result. On a complete non-compact Riemannian manifold, it was independently discovered by Grigor’yan \cite{Gri92} and Saloff-Coste \cite{Sal92,Sal95} that the following two-sided Gaussian estimate of the heat kernel
\begin{equation*}\label{eqn_HK2}\tag*{HK(2)}
\frac{C_1}{V(x,\sqrt{t})}\exp\left(-C_2\frac{d(x,y)^2}{t}\right)\le p_t(x,y)\le\frac{C_3}{V(x,\sqrt{t})}\exp\left(-C_4\frac{d(x,y)^2}{t}\right),
\end{equation*}
where $d(x,y)$ is the geodesic distance, $V(x,r)=m(B(x,r))$ is the Riemannian measure of the open ball $B(x,r)$, is equivalent to the conjunction of the volume doubling condition and the Poincar\'e inequality. However, on many fractals, including the Sierpi\'nski gasket and the Sierpi\'nski carpet, the heat kernel has the following two-sided sub-Gaussian estimate
\begin{align*}
&\frac{C_1}{V(x,t^{1/{d_w}})}\exp\left(-C_2\left(\frac{d(x,y)}{t^{1/d_w}}\right)^{\frac{d_w}{d_w-1}}\right)\nonumber\\
&\le p_t(x,y)\le\frac{C_3}{V(x,t^{1/d_w})}\exp\left(-C_4\left(\frac{d(x,y)}{t^{1/d_w}}\right)^{\frac{d_w}{d_w-1}}\right),\label{eqn_HKdw}\tag*{HK($d_w$)}
\end{align*}
where $d_w$ is a new parameter called the walk dimension, which is always strictly greater than 2 on fractals. For $d_w=2$, \ref{eqn_HKdw} is exactly \ref{eqn_HK2}.

For general $d_w\ge2$, Barlow and Bass \cite{BB04} introduced the well-known cutoff Sobolev inequality and proved that \ref{eqn_HKdw} is equivalent to the conjunction of the volume doubling condition, the Poincar\'e inequality and the cutoff Sobolev inequality. The cutoff Sobolev inequality was later extended to more general settings, see \cite{BBK06,AB15,GHL15,CKW20,CKW21,GHH18}, and the references therein.

Roughly speaking, the cutoff Sobolev inequality provides a family of cutoff functions with controlled energy behavior, but it is significantly stronger than the capacity condition. While the capacity condition is relatively easy to verify on many fractals, it has long been believed that establishing the cutoff Sobolev inequality is complicated and difficult. This is partly due to the fact that existing proofs typically construct the desired cutoff functions using Green's functions with certain estimates. To the best of the author's knowledge, there has been no \emph{direct} verification of the cutoff Sobolev inequality on any fractal.

The main motivation of this paper is to provide a direct proof on the Sierpi\'nski gasket (see Figure \ref{fig_gasket}). The Sierpi\'nski gasket is a prototypical finitely ramified fractal and a standard example of a p.c.f. self-similar set, as described in \cite{Kig01book}, where the so-called compatible condition holds. Our proof would be transparent, and similarly transparent proofs can be carried out for a broad class of fractals satisfying the compatible condition. The technique would rely essentially on the strongly recurrent property of the fractal. Indeed, in \cite{Yan25c}, we applied the same approach to establish a $p$-version of the cutoff Sobolev inequality associated with $p$-energies on metric measure spaces under the so-called slow volume regular condition, which extends the strongly recurrent condition from the case $p=2$ to general $p>1$.

\begin{figure}[ht]
\centering
\includegraphics[width=0.5\linewidth]{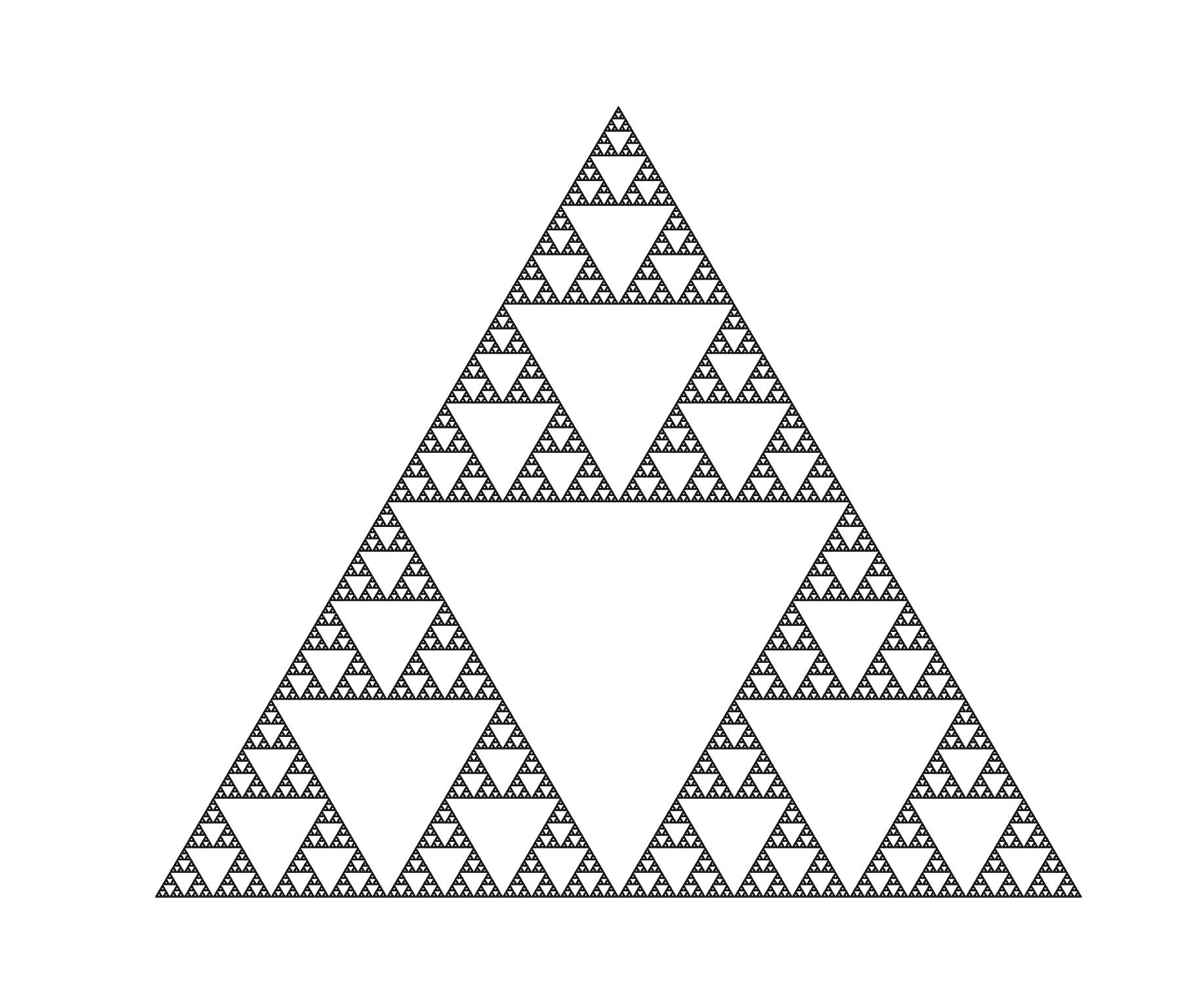}
\caption{The Sierpi\'nski gasket}\label{fig_gasket}
\end{figure}

We introduce the following notation. The letters $C$, $C_1$, $C_2$, $C_A$, $C_B$ will always refer to some positive constants and may change at each occurrence. We use $\#A$ to denote the cardinality of a set $A$. For $a\in\R$, $A\subseteq\R^d$, denote $aA=\{ax:x\in A\}$.

We now present the formal statement of our result. Let $(X,d,m,\calE,\calF)$ be a metric measure Dirichlet space, that is, $(X,d)$ is a locally compact separable metric space, $m$ is a positive Radon measure on $X$ with full support, $(\calE,\calF)$ is a strongly local regular Dirichlet form on $L^2(X;m)$. For any $x\in X$, for any $r\in(0,+\infty)$, denote $B(x,r)=\{y\in X:d(x,y)<r\}$, denote $V(x,r)=m(B(x,r))$. If $B=B(x,r)$, then we denote $\delta B=B(x,\delta r)$ for any $\delta\in(0,+\infty)$. Let $\calB(X)$ be the family of all Borel measurable subsets of $X$. Denote $\dashint_A=\frac{1}{m(A)}\int_A$ for any measurable set $A$ with $m(A)\in(0,+\infty)$. Let $C(X)$ denote the family of all real-valued continuous functions on $X$ and let $C_c(X)$ denote the family of all real-valued continuous functions on $X$ with compact support. Let $\Gamma$ be the energy measure corresponding to the Dirichlet form $(\calE,\calF)$ on $L^2(X;m)$. We refer to \cite{FOT11} for related results about Dirichlet forms.

We say that the volume doubling condition \ref{eqn_VD} holds if there exists $C_{VD}\in(0,+\infty)$ such that
\begin{equation*}\label{eqn_VD}\tag*{VD}
V(x,2r)\le C_{VD}V(x,r)\text{ for any }x\in X,r\in(0,+\infty).
\end{equation*}

Let $U,V$ be two open subsets of $X$ satisfying $U\subseteq\mybar{U}\subseteq V$. We say that $\phi\in\calF$ is a cutoff function for $U\subseteq V$ if $0\le\phi\le1$ $m$-a.e., $\phi=1$ $m$-a.e. in an open neighborhood of $\mybar{U}$ and $\mathrm{supp}(\phi)\subseteq V$, where $\mathrm{supp}(f)$ refers to the support of the measure of $|f|\md m$ for any given function $f$.

We say that the cutoff Sobolev inequality \ref{eqn_CSPsi} holds if there exists $C_S\in(0,+\infty)$ such that for any $x\in X$, for any $R$, $r\in(0,+\infty)$, there exists a cutoff function $\phi\in\calF$ for $B(x,R)\subseteq B(x,R+r)$ such that for any $f\in\calF$, we have
\begin{align*}
&\int_{B(x,R+r)\backslash\mybar{B(x,R)}}|\mytilde{f}|^2\md\Gamma(\phi,\phi)\\
&\le\frac{1}{8}\int_{B(x,R+r)\backslash\mybar{B(x,R)}}|\mytilde{\phi}|^2\md\Gamma(f,f)+\frac{C_S}{\Psi(r)}\int_{B(x,R+r)\backslash\mybar{B(x,R)}}|f|^2\md m,\label{eqn_CSPsi}\tag*{CS($\Psi$)}
\end{align*}
where $\mytilde{f}$, $\mytilde{\phi}$ are quasi-continuous modifications of $f$, $\phi$, respectively, see \cite[Theorem 2.1.3]{FOT11}. For $d_w\in(0,+\infty)$, we say that the cutoff Sobolev inequality \hypertarget{eqn_CSdw}{CS($d_w$)} holds if \ref{eqn_CSPsi} holds with $\Psi:r\mapsto r^{d_w}$.

A simplified version of \ref{eqn_CSPsi} was introduced in \cite[Definition 6.1]{Mur24a} as follows. We say that the simplified cutoff Sobolev inequality \ref{eqn_CSSPsi} holds if there exist $C_1$, $C_2\in(0,+\infty)$, $A\in(1,+\infty)$ such that for any $x\in X$, for any $r\in(0,+\infty)$, there exists a cutoff function $\phi\in\calF$ for $B(x,r)\subseteq B(x,Ar)$ such that for any $f\in\calF$, we have
\begin{equation*}\label{eqn_CSSPsi}\tag*{CSS($\Psi$)}
\int_{B(x,Ar)}|\mytilde{f}|^2\md\Gamma(\phi,\phi)\le C_1\int_{B(x,Ar)} \md\Gamma(f,f)+\frac{C_2}{\Psi(r)}\int_{B(x,Ar)}|f|^2\md m,
\end{equation*}
where $\mytilde{f}$ is a quasi-continuous modification of $f$. For $d_w\in(0,+\infty)$, we say that the simplified cutoff Sobolev inequality \hypertarget{eqn_CSSdw}{CSS($d_w$)} holds if \ref{eqn_CSSPsi} holds with $\Psi:r\mapsto r^{d_w}$.

It is obvious that \ref{eqn_CSPsi} implies \ref{eqn_CSSPsi}. Using the self-improvement property of cutoff Sobolev inequalities, the converse also holds as follows.

\begin{mylem}[{\cite[Lemma 6.2]{Mur24a}}]
Assume that \ref{eqn_VD} and \ref{eqn_CSSPsi} hold. Then \ref{eqn_CSPsi} holds.
\end{mylem}

Our result is a direct proof of the following theorem.

\begin{mythm}\label{thm_main}
The simplified cutoff Sobolev inequality \hyperlink{eqn_CSSdw}{CSS($d_w$)} holds on the Sierpi\'nski gasket, where $d_w=\frac{\log5}{\log2}$ is the walk dimension.
\end{mythm}

\section{Proof}\label{sec_gasket}

For simplicity, we give the proof on the unbounded Sierpi\'nski gasket. In $\R^2$, let
$$p_1=\left(0,0\right),p_2=\left(1,0\right),p_3=\left(\frac{1}{2},\frac{\sqrt{3}}{2}\right),$$
and
$$g_i(x)=\frac{1}{2}\left(x-p_i\right)+p_i,x\in\R^2,i=1,2,3.$$
Then the Sierpi\'nski gasket is the unique non-empty compact set $\calK$ in $\R^2$ satisfying that $\calK=\cup_{i=1}^3g_i(\calK)$, see Figure \ref{fig_gasket}. Let $d_h=\frac{\log3}{\log2}$, then the Hausdorff dimension of $\calK$ is $d_h$ and the $d_h$-dimensional Hausdorff measure $\calH^{d_h}(\calK)\in(0,+\infty)$, see also \cite[Theorem 1.5.7, Proposition 1.5.8]{Kig01book}.

Let $X=\cup_{n=0}^{+\infty}2^n\calK$ be the unbounded Sierpi\'nski gasket, where $\{2^n\calK\}_{n\ge0}$ is an increasing sequence of subsets of $\R^2$. For any $n\in\bbZ$, we say that $K\subseteq X$ is an $n$-cell if $K$ is a translation of $2^n\calK$. Let $d$ be the Euclidean metric on $X$, then each $n$-cell has diameter $2^n$. Let $m$ be the unique positive Radon measure on $X$ such that each $n$-cell has measure $3^n$, indeed, $m=\frac{1}{\calH^{d_h}(\calK)}\calH^{d_h}$ on $X$. It is obvious that $(X,d,m)$ is an unbounded metric measure space.

For any $n\in\bbZ$, for any $n$-cell $K$, let $\calN(K)$ denote the $n$-cell neighborhood of $K$ in $X$, defined by
$$\calN(K)=\bigcup_{\mytilde{K}:n\text{-cell},\mytilde{K}\cap K\ne\emptyset}\mytilde{K},$$
then $\calN(K)\supseteq K$ and
\begin{equation*}
\#\{\mytilde{K}:n\text{-cell},\mytilde{K}\subseteq\calN(K)\}=
\begin{cases}
3&\text{if }K=2^n\calK,\\
4&\text{otherwise}.
\end{cases}
\end{equation*}

Let $V_0=\{p_1,p_2,p_3\}$ and $V_{n+1}=\cup_{i=1}^3g_i(V_n)$ for any $n\ge0$, then $\{2^nV_n\}_{n\ge0}$ is an increasing sequence of finites subsets of $\R^2$. Let $\calV_0=\cup_{n\ge0}2^nV_n$ and $\calV_n=2^n\calV_0$ for any $n\in\bbZ$. For any $n\in\bbZ$, for each $n$-cell $K$, the closed convex hull of $K$ in $\R^2$ is an equilateral triangle. The set of its three vertices is exactly $K\cap\calV_n$, and we refer to these three vertices as the boundary points of $K$.

Let
\begin{align*}
&\calE(u,u)=\lim_{n\to+\infty}\left(\frac{5}{3}\right)^n\sum_{p,q\in\calV_{-n},d(p,q)=2^{-n}}(u(p)-u(q))^2,\\
&\calF=\left\{u\in L^2(X;m)\cap C(X):\lim_{n\to+\infty}\left(\frac{5}{3}\right)^n\sum_{p,q\in\calV_{-n},d(p,q)=2^{-n}}(u(p)-u(q))^2<+\infty\right\}.
\end{align*}
By the standard theory in \cite{Kig01book,Str06book}, the above sequence $\left(\frac{5}{3}\right)^n\sum_{p,q:\ldots}(u(p)-u(q))^2$ is monotone increasing in $n$, and $(\calE,\calF)$ is a strongly local regular Dirichlet form on $L^2(X;m)$.

Let $d_w=\frac{\log5}{\log2}$. We have the following Morrey-Sobolev inequality. The proof is standard and relies on the geometric properties of the Sierpi\'nski gasket, and is therefore omitted.

\begin{mylem}[Morrey-Sobolev inequality, {\cite[LEMMA 2.4]{FH99}}]\label{lem_SG_MS}
There exists $C_{MS}\in(0,+\infty)$ such that for any $n\in\bbZ$, for any $n$-cell $K$, for any $x,y\in\calN(K)$, for any $f\in\calF$, we have
$$|f(x)-f(y)|^2\le C_{MS}d(x,y)^{d_w-d_h}\int_{\calN(K)}\md\Gamma(f,f),$$
where
$$\int_{\calN(K)}\md\Gamma(f,f)=\lim_{m\to+\infty}\left(\frac{5}{3}\right)^{m}\sum_{p,q\in\calN(K)\cap\calV_{-m},d(p,q)=2^{-m}}(f(p)-f(q))^2.$$
\end{mylem}

We have the following result concerning cutoff functions on cells.

\begin{mylem}\label{lem_SG_cell}
For any $n\in\bbZ$, for any $n$-cell $K$, there exists $\phi_K\in\calF\cap C_c(X)$ with $0\le\phi_K\le1$ in $X$, $\phi_K=1$ on $K$, $\phi_K=0$ on $X\backslash \calN(K)$ such that for any $f\in\calF$, we have
$$\int_{\calN(K)}|f|^2\md\Gamma(\phi_K,\phi_K)\le\frac{200C_{MS}}{3}\int_{\calN(K)}\md\Gamma(f,f)+24 \left(\frac{3}{5}\right)^n\dashint_{\calN(K)}|f|^2\md m,$$
where $C_{MS}$ is the constant in Lemma \ref{lem_SG_MS}.
\end{mylem}

\begin{proof}
Without loss of generality, we consider the case $\#\{\mytilde{K}:n\text{-cell},\mytilde{K}\subseteq\calN(K)\}=4$. Write $\calN(K)=K\cup K_1\cup K_2\cup K_3$, let $x_1$, $x_2$, $x_3$ be the boundary points of $K$, and $x_i$, $y_i$, $z_i$ the boundary points of $K_i$ for $i=1,2,3$, see Figure \ref{fig_SG_nbhd}.

\begin{figure}[ht]
\centering
\begin{tikzpicture}[scale=3]

\coordinate (A) at (0,0);
\coordinate (B) at (2,0);
\coordinate (C) at (1,2*0.8660254037);

\coordinate (D) at ($(B)!1/2!(C)$);
\coordinate (E) at ($(A)!1/2!(C)$);
\coordinate (F) at ($(A)!1/2!(B)$);

\coordinate (AF) at ($(A)!1/2!(F)$);
\coordinate (AE) at ($(A)!1/2!(E)$);
\coordinate (EF) at ($(E)!1/2!(F)$);

\coordinate (BD) at ($(B)!1/2!(D)$);
\coordinate (BF) at ($(B)!1/2!(F)$);
\coordinate (DF) at ($(D)!1/2!(F)$);

\coordinate (CD) at ($(C)!1/2!(D)$);
\coordinate (CE) at ($(C)!1/2!(E)$);
\coordinate (DE) at ($(D)!1/2!(E)$);

\coordinate (DEE) at ($(DE)!1/2!(E)$);
\coordinate (K1) at ($(DEE)!1/3!(CE)$);

\coordinate (AEF) at ($(AE)!1/2!(EF)$);
\coordinate (K) at ($(AEF)!1/3!(E)$);

\coordinate (AAF) at ($(A)!1/2!(AF)$);
\coordinate (K2) at ($(AAF)!1/3!(AE)$);

\coordinate (AFF) at ($(AF)!1/2!(F)$);
\coordinate (K3) at ($(AFF)!1/3!(EF)$);

\draw (A) -- (B) -- (C) -- cycle;

\draw (D) -- (E) -- (F) -- cycle;

\draw (CD) -- (DE) -- (CE) -- cycle;
\draw (AF) -- (AE) -- (EF) -- cycle;
\draw (BD) -- (BF) -- (DF) -- cycle;

\node at (K) {$K$};
\node at (K1) {$K_1$};
\node at (K2) {$K_2$};
\node at (K3) {$K_3$};

\node[left] at (E) {$x_1$};
\node[left] at (CE) {$y_1$};
\node[left] at (AE) {$x_2$};
\node[below left] at (A) {$y_2$};
\node[below] at (DE) {$z_1$};
\node[right] at (EF) {$x_3$};
\node[below] at (AF) {$z_2=z_3$};
\node[below] at (F) {$y_3$};

\filldraw[black] (A) circle (0.02);
\filldraw[black] (AF) circle (0.02);
\filldraw[black] (F) circle (0.02);
\filldraw[black] (EF) circle (0.02);
\filldraw[black] (AE) circle (0.02);
\filldraw[black] (E) circle (0.02);
\filldraw[black] (DE) circle (0.02);
\filldraw[black] (CE) circle (0.02);

\end{tikzpicture}
\caption{$\calN(K)$}\label{fig_SG_nbhd}
\end{figure}

Let $\phi_K$ be given by $\phi_K=1$ on $K$, $\phi_K=0$ on $X\backslash\calN(K)$, and on each $K_i$, $\phi_K$ is the harmonic function on $K_i$ with $\phi_K(x_i)=1$ and $\phi_K(y_i)=\phi_K(z_i)=0$, which is given by the standard $\frac{2}{5}$-$\frac{2}{5}$-$\frac{1}{5}$-algorithm, see \cite{Kig01book,Str06book}. Then $\phi_K\in\calF\cap C_c(X)$ is well-defined and
\begin{align*}
&\calE(\phi_K,\phi_K)=\int_{\calN(K)}\Gamma(\phi_K,\phi_K)\\
&=\lim_{m\to+\infty}\left(\frac{5}{3}\right)^{m}\sum_{p,q\in \calN(K)\cap\calV_{-m},d(p,q)=2^{-m}}\left(\phi_K(p)-\phi_K(q)\right)^2\\
&=\sum_{i=1}^3\lim_{-n\le m\to+\infty}\left(\frac{5}{3}\right)^{m}\sum_{p,q\in K_i\cap\calV_{-m},d(p,q)=2^{-m}}\left(\phi_K(p)-\phi_K(q)\right)^2\\
&=\sum_{i=1}^3\left(\frac{5}{3}\right)^{-n}\sum_{p,q\in K_i\cap\calV_{n},d(p,q)=2^{n}}\left(\phi_K(p)-\phi_K(q)\right)^2=6\left(\frac{3}{5}\right)^n.
\end{align*}

For any $f\in\calF\subseteq C(X)$, there exists $x_0\in\calN(K)$ such that $|f(x_0)|^2=\dashint_{\calN(K)}|f|^2\md m$. By the definition of energy measures, we have
\begin{align*}
&\int_{\calN(K)}|f|^2\md\Gamma(\phi_K,\phi_K)\\
&=2\lim_{m\to+\infty}\left(\frac{5}{3}\right)^{m}\sum_{p,q\in\calN(K)\cap\calV_{-m},d(p,q)=2^{-m}}\left((\phi_Kf^2)(p)-(\phi_Kf^2)(q)\right)\left(\phi_K(p)-\phi_K(q)\right)\\
&\hspace{10pt}-\lim_{m\to+\infty}\left(\frac{5}{3}\right)^{m}\sum_{p,q\in\calN(K)\cap\calV_{-m},d(p,q)=2^{-m}}\left((\phi_K^2)(p)-(\phi_K^2)(q)\right)\left((f^2)(p)-(f^2)(q)\right)\\
&=\lim_{m\to+\infty}\left(\frac{5}{3}\right)^{m}\sum_{p,q\in\calN(K)\cap\calV_{-m},d(p,q)=2^{-m}}\left(f(p)^2+f(q)^2\right)\left(\phi_K(p)-\phi_K(q)\right)^2.
\end{align*}
For any $p\in\calN(K)$, by Lemma \ref{lem_SG_MS}, we have
$$|f(p)-f(x_0)|^2\le C_{MS}d(p,x_0)^{d_w-d_h}\int_{\calN(K)}\md\Gamma(f,f)\le C_{MS}2^{(d_w-d_h)(n+2)}\int_{\calN(K)}\md\Gamma(f,f),$$
hence
$$|f(p)|^2\le2|f(p)-f(x_0)|^2+2|f(x_0)|^2\le2 C_{MS}2^{(d_w-d_h)(n+2)}\int_{\calN(K)}\md\Gamma(f,f)+2\dashint_{\calN(K)}|f|^2\md m,$$
which gives
\begin{align*}
&\int_{\calN(K)}|f|^2\md\Gamma(\phi_K,\phi_K)\\
&\le2\left(2 C_{MS}2^{(d_w-d_h)(n+2)}\int_{\calN(K)}\md\Gamma(f,f)+2\dashint_{\calN(K)}|f|^2\md m\right)\\
&\hspace{10pt}\cdot\lim_{m\to+\infty}\left(\frac{5}{3}\right)^{m}\sum_{p,q\in\calN(K)\cap\calV_{-m},d(p,q)=2^{-m}}\left(\phi_K(p)-\phi_K(q)\right)^2\\
&=2\left(2 C_{MS}2^{(d_w-d_h)(n+2)}\int_{\calN(K)}\md\Gamma(f,f)+2\dashint_{\calN(K)}|f|^2\md m\right)\cdot6\left(\frac{3}{5}\right)^n\\
&=\frac{200C_{MS}}{3}\int_{\calN(K)}\md\Gamma(f,f)+24 \left(\frac{3}{5}\right)^n\dashint_{\calN(K)}|f|^2\md m.
\end{align*}
\end{proof}

We have \hyperlink{eqn_CSSdw}{CSS($d_w$)} holds as follows.

\begin{myprop}\label{prop_SG_ball}
For any ball $B(x_0,r)$, there exists $\phi\in\calF\cap C_c(X)$ with $0\le\phi\le1$ in $X$, $\phi=1$ in $B(x_0,r)$, $\phi=0$ on $X
\backslash B(x_0,8r)$ such that for any $f\in\calF$, we have
$$\int_{B(x_0,8r)}|f|^2\md\Gamma(\phi,\phi)\le\frac{800C_{MS}}{3}\int_{B(x_0,8r)}\md\Gamma(f,f)+\frac{96}{r^{d_w}}\int_{B(x_0,8r)}|f|^2\md m,$$
where $C_{MS}$ is the constant in Lemma \ref{lem_SG_MS}.
\end{myprop}

\begin{proof}
Let $n$ be the integer satisfying $2^{n-1}\le r<2^n$, there exists an $n$-cell $K$ such that $x_0\in K$, then $B(x_0,r)\subseteq\calN(K)$, see Figure \ref{fig_SG_nbhd}. For any $n$-cell $\mytilde{K}\subseteq\calN(K)$, by Lemma \ref{lem_SG_cell}, there exists $\phi_{\mytilde{K}}\in\calF\cap C_c(X)$ with $0\le\phi_{\mytilde{K}}\le1$ in $X$, $\phi_{\mytilde{K}}=1$ on $\mytilde{K}$, $\phi_{\mytilde{K}}=0$ on $X\backslash\calN(\mytilde{K})$ such that for any $f\in\calF$, we have
$$\int_{\calN(\mytilde{K})}|f|^2\md\Gamma(\phi_{\mytilde{K}},\phi_{\mytilde{K}})\le\frac{200C_{MS}}{3}\int_{\calN({\mytilde{K}})}\md\Gamma(f,f)+24 \left(\frac{3}{5}\right)^n\dashint_{\calN({\mytilde{K}})}|f|^2\md m.$$
Let
$$\phi=\max_{\mytilde{K}:n\text{-cell},\mytilde{K}\subseteq\calN(K)}\phi_{\mytilde{K}},$$
then $\phi\in\calF\cap C_c(X)$ satisfies $0\le\phi\le1$ in $X$, $\phi=1$ on $\calN(K)\supseteq B(x_0,r)$, $\phi=0$ on $X\backslash\cup_{{\mytilde{K}:n\text{-cell},\mytilde{K}\subseteq\calN(K)}}\calN(\mytilde{K})$. Since $\calN(\mytilde{K})\subseteq B(x_0,2^{n+2})\subseteq B(x_0,8r)$ for any $n$-cell $\mytilde{K}\subseteq\calN(K)$, we have $\phi=0$ on $X\backslash B(x_0,8r)\subseteq X\backslash\cup_{{\mytilde{K}:n\text{-cell},\mytilde{K}\subseteq\calN(K)}}\calN(\mytilde{K})$. By the Markovian property of Dirichlet forms, for any $A\in\calB(X)$, we have
\begin{equation}\label{eqn_Markov}
\Gamma(\phi,\phi)(A)\le\sum_{\mytilde{K}:n\text{-cell},\mytilde{K}\subseteq\calN(K)}\Gamma(\phi_{\mytilde{K}},\phi_{\mytilde{K}})(A).
\end{equation}
Hence for any $f\in\calF$, we have
\begin{align*}
&\int_{B(x_0,8r)}|f|^2\md\Gamma(\phi,\phi)\le\sum_{\mytilde{K}:n\text{-cell},\mytilde{K}\subseteq\calN(K)}\int_{B(x_0,8r)}|f|^2\md\Gamma(\phi_{\mytilde{K}},\phi_{\mytilde{K}})\\
&=\sum_{\mytilde{K}}\int_{\calN(\mytilde{K})}|f|^2\md\Gamma(\phi_{\mytilde{K}},\phi_{\mytilde{K}})\le\sum_{\mytilde{K}}\left(\frac{200C_{MS}}{3}\int_{\calN({\mytilde{K}})}\md\Gamma(f,f)+24 \left(\frac{3}{5}\right)^n\dashint_{\calN({\mytilde{K}})}|f|^2\md m\right)\\
&\overset{\calN(\mytilde{K})\subseteq B(x_0,8r)}{\underset{m(\calN(\mytilde{K}))\ge m(\mytilde{K})=3^n}{\scalebox{15}[1]{$\le$}}}\sum_{\mytilde{K}}\left(\frac{200C_{MS}}{3}\int_{B(x_0,8r)}\md\Gamma(f,f)+\frac{24}{5^n}\int_{B(x_0,8r)}|f|^2\md m\right)\\
&\overset{\#\{\mytilde{K}:n\text{-cell},\mytilde{K}\subseteq\calN(K)\}\le4}{\underset{5^{n-1}\le r^{d_w}<5^n}{\scalebox{15}[1]{$\le$}}}4\cdot\left(\frac{200C_{MS}}{3}\int_{B(x_0,8r)}\md\Gamma(f,f)+\frac{24}{r^{d_w}}\int_{B(x_0,8r)}|f|^2\md m\right)\\
&=\frac{800C_{MS}}{3}\int_{B(x_0,8r)}\md\Gamma(f,f)+\frac{96}{r^{d_w}}\int_{B(x_0,8r)}|f|^2\md m.
\end{align*}
\end{proof}

\bibliographystyle{plain}

\begin{thebibliography}{10}

\bibitem{AB15}
Sebastian Andres and Martin~T. Barlow.
\newblock Energy inequalities for cutoff functions and some applications.
\newblock {\em J. Reine Angew. Math.}, 699:183--215, 2015.

\bibitem{BB04}
Martin~T. Barlow and Richard~F. Bass.
\newblock Stability of parabolic {H}arnack inequalities.
\newblock {\em Trans. Amer. Math. Soc.}, 356(4):1501--1533, 2004.

\bibitem{BBK06}
Martin~T. Barlow, Richard~F. Bass, and Takashi Kumagai.
\newblock Stability of parabolic {H}arnack inequalities on metric measure
  spaces.
\newblock {\em J. Math. Soc. Japan}, 58(2):485--519, 2006.

\bibitem{CKW20}
Zhen-Qing Chen, Takashi Kumagai, and Jian Wang.
\newblock Stability of parabolic {H}arnack inequalities for symmetric non-local
  {D}irichlet forms.
\newblock {\em J. Eur. Math. Soc. (JEMS)}, 22(11):3747--3803, 2020.

\bibitem{CKW21}
Zhen-Qing Chen, Takashi Kumagai, and Jian Wang.
\newblock Stability of heat kernel estimates for symmetric non-local
  {D}irichlet forms.
\newblock {\em Mem. Amer. Math. Soc.}, 271(1330):v+89, 2021.

\bibitem{FH99}
Kenneth~J. Falconer and Jiaxin Hu.
\newblock Non-linear elliptical equations on the {S}ierpi\'nski gasket.
\newblock {\em J. Math. Anal. Appl.}, 240(2):552--573, 1999.

\bibitem{FOT11}
Masatoshi Fukushima, Yoichi Oshima, and Masayoshi Takeda.
\newblock {\em Dirichlet forms and symmetric {M}arkov processes}, volume~19 of
  {\em De Gruyter Studies in Mathematics}.
\newblock Walter de Gruyter \& Co., Berlin, extended edition, 2011.

\bibitem{Gri92}
Alexander Grigor'yan.
\newblock The heat equation on noncompact {R}iemannian manifolds.
\newblock {\em Mat. Sb.}, 182(1):55--87, 1991.

\bibitem{GHH18}
Alexander Grigor'yan, Eryan Hu, and Jiaxin Hu.
\newblock Two-sided estimates of heat kernels of jump type {D}irichlet forms.
\newblock {\em Adv. Math.}, 330:433--515, 2018.

\bibitem{GHL15}
Alexander Grigor'yan, Jiaxin Hu, and Ka-Sing Lau.
\newblock Generalized capacity, {H}arnack inequality and heat kernels of
  {D}irichlet forms on metric measure spaces.
\newblock {\em J. Math. Soc. Japan}, 67(4):1485--1549, 2015.

\bibitem{Kig01book}
Jun Kigami.
\newblock {\em Analysis on fractals}, volume 143 of {\em Cambridge Tracts in
  Mathematics}.
\newblock Cambridge University Press, Cambridge, 2001.

\bibitem{Mur24a}
Mathav Murugan.
\newblock Heat kernel for reflected diffusion and extension property on uniform
  domains.
\newblock {\em Probab. Theory Related Fields}, 190(1-2):543--599, 2024.

\bibitem{Sal92}
Laurent Saloff-Coste.
\newblock A note on {P}oincar\'{e}, {S}obolev, and {H}arnack inequalities.
\newblock {\em Internat. Math. Res. Notices}, 1992(2):27--38, 1992.

\bibitem{Sal95}
Laurent Saloff-Coste.
\newblock Parabolic {H}arnack inequality for divergence-form second-order
  differential operators.
\newblock {\em Potential Anal.}, 4(4):429--467, 1995.
\newblock Potential theory and degenerate partial differential operators
  (Parma).

\bibitem{Str06book}
Robert~S. Strichartz.
\newblock {\em Differential equations on fractals}.
\newblock Princeton University Press, Princeton, NJ, 2006.
\newblock A tutorial.

\bibitem{Yan25c}
Meng {Yang}.
\newblock {On singularity of $p$-energy measures on metric measure spaces}.
\newblock {\em arXiv e-prints}, page arXiv:2505.12468v1, May 2025.

\end{thebibliography}

\end{document}